\documentclass[11pt]{amsart}

\usepackage[T1]{fontenc}
\usepackage{amsmath,amssymb,mathtools}
\usepackage{microtype}
\usepackage[hidelinks]{hyperref}

\theoremstyle{plain}
\newtheorem{theorem}{Theorem}
\newtheorem{proposition}[theorem]{Proposition}
\newtheorem{lemma}[theorem]{Lemma}
\theoremstyle{remark}
\newtheorem{remark}[theorem]{Remark}

\newcommand{\F}{\mathbb F}
\newcommand{\La}{\operatorname{La}}
\DeclareMathOperator{\Span}{span}

\title[An Improved Lower Bound for Diamond-Free Families]{An Improved Lower Bound for Diamond-Free Families}
\author{Casey Tompkins}
\email{casey.tompkins@renyi.hu}
\date{}

\makeatletter
\def\@setauthors{%
  \begingroup
  \def\thanks{\protect\thanks@warning}%
  \trivlist
  \centering\scriptsize \@topsep30\p@\relax
  \advance\@topsep by -\baselineskip
  \item\relax
  \author@andify\authors
  \def\\{\protect\linebreak}%
  \MakeUppercase{\authors}%
  \ifx\@empty\contribs
  \else
    ,\penalty-3 \space \@setcontribs
    \@closetoccontribs
  \fi
  \endtrivlist
  \endgroup
}
\makeatother

\begin{document}

\begin{abstract}
We construct a diamond-free family in the Boolean lattice whose size is asymptotically larger than the union of two middle layers. Denote the diamond poset by $Q_2$, and let $\La(n,Q_2)$ be the maximum size of a family in $2^{[n]}$ containing no weak copy of $Q_2$. We prove
\[
\La(n,Q_2)
\ge
\left(2+\frac{29}{135}\prod_{i=1}^{\infty}(1-4^{-i})+o(1)\right)
\binom{n}{\lfloor n/2\rfloor}.
\]
The leading coefficient is approximately $2.147908$. In particular, this disproves the diamond conjecture.

\end{abstract}

\maketitle
\vspace{-0.9\baselineskip}

\section{Introduction}
Let $[n]=\{1,2,\dots,n\}$ and let $2^{[n]}$ denote the family of subsets of $[n]$ ordered by inclusion.  For a finite poset~$P$, a family $\mathcal F\subseteq 2^{[n]}$ contains a weak copy of~$P$ if there is an injective map $\psi\colon P\to \mathcal{F}$ such that, whenever $p<q$ in $P$, we have $\psi(p) \subset \psi(q)$.  We write
\[
\La(n,P)=\max\{|\mathcal F| : \mathcal F\subseteq 2^{[n]}\text{ contains no weak copy of }P\}.
\]

These problems extend the classical theorem of Sperner~\cite{Sperner}, which is the case when $P$ is the poset consisting of two comparable elements.  Erd\H{o}s's theorem~\cite{Erdos} determining the maximum size of a family avoiding $k+1$ distinct sets $A_1,A_2,\dots,A_{k+1}$ with $A_1\subset A_2 \subset \dots \subset A_{k+1}$
corresponds to taking $P$ to be a total order of size $k+1$.  The extremal constructions are the unions of $k$ consecutive layers of sizes closest to $n/2$.

Katona and Tarj\'an~\cite{KatonaTarjan} initiated the general study of the function $\La(n,P)$.  Their first examples already included the three-element poset $V$ with relations $p<q$ and $p<r$.  They proved asymptotic results for this case, and also gave exact results for avoiding $V$ together with its dual $\Lambda$. This line of work was subsequently extended to $r$-forks and other small height-two posets, for example by De Bonis and Katona~\cite{DeBonisKatona}, and then to a wide variety of further posets.

For a general finite poset $P$, let $e(P)$ be the largest integer $t$ such that the union of any $t$ consecutive layers of any Boolean lattice is $P$-free.  This gives the general lower bound $\La(n,P)\ge (e(P)+o(1))\binom{n}{\lfloor n/2\rfloor}$.  Bukh~\cite{Bukh} proved the matching asymptotic for every poset whose Hasse diagram is a tree.  If such a poset $T$ has height $h(T)$, then $\La(n,T)=(h(T)-1+o(1))\binom{n}{\lfloor n/2\rfloor}$.  Bukh~\cite{Bukh} and Griggs and Lu~\cite{GriggsLu} conjectured that, for every fixed finite poset~$P$,
\[
\La(n,P)=(e(P)+o(1))\binom{n}{\lfloor n/2\rfloor}.
\]

This conjecture was consistent with the shape of the known constructions, which mostly had the form of consecutive middle layers, sometimes supplemented by scattered sets from one additional layer.  Ellis, Ivan and Leader~\cite{EllisIvanLeader} disproved the conjecture in general using a construction based on daisy-free hypergraphs of positive Tur\'an density.  In particular, they disproved the conjecture for $P=Q_d$ with $d\ge 4$, where $Q_d$ is the Boolean lattice of dimension~$d$.
The construction of the daisy-free layer involves labelling elements by vectors in a finite-dimensional vector space and keeping linearly independent sets to obtain a daisy-free family in one middle layer.
In the resulting construction, this intermediate layer replaces one full layer inside a block of otherwise full layers.  They note that their construction leaves the diamond case $Q_2$ open.  Recently, Gerbner and Patk\'os~\cite{GerbnerPatkos} obtained a counterexample to the conjecture of Bukh~\cite{Bukh} and Griggs and Lu~\cite{GriggsLu} for a height-two poset~$P$ with $e(P)=1$, also using daisy-free hypergraphs.

Determining the value of $\La(n,Q_2)$ is a central open case in the forbidden poset literature. A widely held conjecture, often referred to as the diamond conjecture (see, for example, Johnston and Lu~\cite{JohnstonLu}), asserts that
\[
\La(n,Q_2)=(2+o(1))\binom{n}{\lfloor n/2\rfloor}.
\]
Since $e(Q_2)=2$, this agrees with the prediction of the (now disproved) conjecture mentioned earlier.  Before the present construction, no lower bound with leading coefficient exceeding $2$ was known for $\La(n,Q_2)$.

One indication that the coefficient $2$ for the diamond was not tied to two complete layers was the work of Czabarka, Dutle, Johnston and Sz\'ekely~\cite{CzabarkaDutleJohnstonSzekely}. They produced  diamond-free families based on abelian-groups using more than $2$ layers, all of which have density strictly between $0$ and $1$, while still having total size  $(2+o(1))\binom{n}{\lfloor n/2\rfloor}$. They note, however, that their method cannot produce constructions of leading coefficient greater than $2$.

Our main result is the following lower bound.

\begin{theorem}\label{thmMain}
With
\[
\Delta_4=\prod_{i=1}^{\infty}(1-4^{-i}),
\]
we have
\[
\La(n,Q_2)\ge
\left(2+\frac{29}{135}\Delta_4+o(1)\right)
\binom{n}{\lfloor n/2\rfloor}.
\]
The coefficient is approximately $2.147908$.
\end{theorem}

Since $e(Q_2)=2$, Theorem~\ref{thmMain} gives a counterexample to the equality $\La(n,P)=(e(P)+o(1))\binom{n}{\lfloor n/2\rfloor}$ already for $P=Q_2$.  

The upper bound for $\La(n,Q_2)$ has been successively improved across several papers.  The trivial coefficient $3$ follows from Erd\H{o}s's theorem~\cite{Erdos}, since a $Q_2$-free family contains no chain of four sets.  Griggs, Li and Lu~\cite{GriggsLiLu} used a chain-counting argument to obtain the coefficient $2.5$, and then refined the upper bound to $2.296$.  Axenovich, Manske and Martin~\cite{AxenovichManskeMartin} improved this to $2.283261$. Griggs, Li and Lu~\cite{GriggsLiLu} then improved it to $25/11\approx2.273$.  Kramer, Martin and Young~\cite{KramerMartinYoung} obtained $2.25$.  The best known upper bound is $(3+\sqrt2)/2\approx2.2071$ due to Gr\'osz, Methuku and Tompkins~\cite{GroszMethukuTompkins}.  For families using at most three layers, Manske and Shen~\cite{ManskeShen} proved the upper bound of $((3+2\sqrt 3)/3+o(1))\binom{n}{\lfloor n/2\rfloor}$ under this restriction, where $(3+2\sqrt3)/3\approx2.1547$.  Under the same three-layer restriction, Balogh, Hu, Lidick\'y and Liu~\cite{BaloghHuLidickyLiu} used flag algebras to improve the leading coefficient to approximately $2.15121$.  Thus the coefficient in Theorem~\ref{thmMain} is only approximately $0.003302$ below the best known upper bound from the three-layer setting.

Our construction uses only the layers of sizes $k-1,k,k+1$, where $k=\lfloor n/2\rfloor$.  Our approach, like that of Ellis, Ivan and Leader~\cite{EllisIvanLeader}, is based on linear algebra, but we impose conditions on the dimensions of the spans of the labels in all three layers.
After labelling $[n]$ by vectors in $\F_4^k$, we retain lower-layer sets whose label spans have dimension $k-1$, middle-layer sets whose label spans have dimension at most $k-1$, and upper-layer sets whose label spans do not have dimension $k-1$.  If a retained lower-layer set $B$ has two retained one-element extensions, both added labels lie in the span of the labels of $B$.  Consequently, the labels of their union span the same $(k-1)$-dimensional subspace, so the union is excluded from the upper layer.  The limiting densities in the three layers are approximately $0.91805$, $0.31146$, and $0.91840$, respectively.

\section{Proof of Theorem~\ref{thmMain}}

Unless otherwise specified, all vector spaces considered are over the field $\F_4$.  Let
\[
k=\lfloor n/2\rfloor,
\qquad
V=\F_4^k,
\]
and let $\phi\colon[n]\to V$ be a labelling of the ground set.  For $A\subseteq[n]$, let
\[
W_\phi(A)=\Span\{\phi(a) : a\in A\},
\]
which we call the label span of $A$.  When the labelling is clear, we write $W(A)$.  Equivalently, if the vectors $\phi(a)$ with $a\in A$ are regarded as the columns of a $k\times |A|$ matrix, then $W_\phi(A)$ is its column space and $\dim W_\phi(A)$ is its rank.  We use only the elementary fact that adjoining a vector already in a span leaves the span unchanged, while adjoining a vector outside it increases its dimension by one.

Define
\[
\mathcal L_\phi=\{A\in \binom{[n]}{k-1} : \dim W(A)=k-1\},
\]
\[
\mathcal M_\phi=\{A\in \binom{[n]}{k} : \dim W(A)\le k-1\},
\]
and
\[
\mathcal U_\phi=\{A\in \binom{[n]}{k+1} : \dim W(A)\ne k-1\}.
\]
Finally, set
\[
\mathcal F_\phi=\mathcal L_\phi\cup\mathcal M_\phi\cup\mathcal U_\phi.
\]
Thus the lower layer consists of sets whose labels span a $(k-1)$-dimensional subspace, the middle layer consists of sets whose labels span a subspace of dimension at most $k-1$, and the upper layer omits precisely the sets whose labels span a $(k-1)$-dimensional subspace.

\begin{proposition}\label{propQfree}
For every labelling $\phi\colon[n]\to\F_4^k$, the family $\mathcal F_\phi$ is $Q_2$-free.
\end{proposition}

\begin{proof}
Suppose that $B,X,Y,C\in\mathcal F_\phi$ form a weak copy of $Q_2$, with
\[
B\subset X\subset C,
\qquad
B\subset Y\subset C,
\qquad
X\ne Y.
\]
Since $\mathcal F_\phi$ is supported on the three consecutive layers of set sizes $k-1,k,k+1$, we must have
\[
|B|=k-1,
\qquad
|X|=|Y|=k,
\qquad
|C|=k+1.
\]
Hence $B\in\mathcal L_\phi$, $X,Y\in\mathcal M_\phi$, and $C\in\mathcal U_\phi$.  There are distinct $x,y\notin B$ such that
\[
X=B\cup\{x\},
\qquad
Y=B\cup\{y\},
\qquad
C=B\cup\{x,y\}.
\]
Since $B\in\mathcal L_\phi$, we have $\dim W(B)=k-1$.  Since $X\in\mathcal M_\phi$, we have $\dim W(X)\le k-1$, while $B\subset X$ implies $W(B)\subseteq W(X)$.  Hence $W(X)=W(B)$, and in particular
\[
\phi(x)\in W(B).
\]
The same argument applied to $Y$ gives
\[
\phi(y)\in W(B).
\]
Thus $W(C)=W(B)$, so
\[
\dim W(C)=k-1.
\]
This contradicts the definition of $\mathcal U_\phi$.  Hence no such copy of $Q_2$ exists.
\end{proof}
For each $t \in [n]$, choose $\phi(t)$ independently and uniformly at random
from~$V$.  For a set $A$ of size $s$, the labels indexed by $A$ are $s$ vectors sampled independently and uniformly from $\F_4^k$.

Recall that the Gaussian binomial coefficient
\[
\genfrac{[}{]}{0pt}{}{m}{r}_q
=\prod_{i=0}^{r-1}\frac{q^m-q^i}{q^r-q^i}
\]
is the number of $r$-dimensional subspaces of $\F_q^m$.

We use the following well-known formula for counting matrices of a given rank, and include a proof for completeness.

\begin{proposition}\label{propRankcount}
Let $q$ be a prime power and let $0\le r\le \min\{m,s\}$.  The number of $m\times s$ matrices over $\F_q$ having rank $r$ is
\begin{equation}\label{eqRankcount}
\genfrac{[}{]}{0pt}{}{m}{r}_q\prod_{i=0}^{r-1}(q^s-q^i).
\end{equation}
\end{proposition}

\begin{proof}
First choose the column space $W$, which may be any $r$-dimensional subspace of $\F_q^m$.  There are $\genfrac{[}{]}{0pt}{}{m}{r}_q$ choices.  Once $W$ is fixed, fix a basis of $W$ and write the columns in this basis.  The resulting $r\times s$ coordinate matrix must have rank $r$.  By equality of row rank and column rank, this is equivalent to its $r$ rows being linearly independent vectors in $\F_q^s$.  Choosing the rows successively gives
\[
(q^s-1)(q^s-q)\cdots(q^s-q^{r-1})
=\prod_{i=0}^{r-1}(q^s-q^i)
\]
possibilities.  Multiplying by the number of choices for $W$ proves \eqref{eqRankcount}.
\end{proof}

Taking $r=s\le m$ in Proposition~\ref{propRankcount} and dividing by the total number $q^{ms}$ of matrices shows that vectors $v_1,\dots,v_s$ chosen independently and uniformly from $\F_q^m$ are linearly independent with probability $\prod_{i=0}^{s-1}(1-q^{i-m})$.

\begin{lemma}\label{lemDensities}
For the random labelling above, the following estimates hold for sets of the indicated sizes.
\begin{align}
&\Pr(A\in\mathcal L_\phi)=\frac43\Delta_4+o(1), && |A|=k-1, \label{eqLdensity}\\
&\Pr(A\in\mathcal M_\phi)=1-\Delta_4+o(1), && |A|=k, \label{eqMdensity}\\
&\Pr(A\in\mathcal U_\phi)=1-\frac{16}{135}\Delta_4+o(1), && |A|=k+1. \label{eqUdensity}
\end{align}
\end{lemma}

\begin{proof}
For $|A|=k-1$, the preceding probability formula gives
\[
\Pr(A\in\mathcal L_\phi)=\prod_{i=0}^{k-2}(1-4^{i-k})
=\prod_{j=2}^{k}(1-4^{-j})
=\frac43\Delta_4+o(1).
\]
For $|A|=k$, the set $A$ belongs to $\mathcal M_\phi$ precisely when $\dim W(A)\le k-1$, so this event is the complement of the event that the $k$ labels are linearly independent.  Therefore
\[
\Pr(A\in\mathcal M_\phi)=1-\prod_{i=0}^{k-1}(1-4^{i-k})
=1-\prod_{j=1}^{k}(1-4^{-j})
=1-\Delta_4+o(1).
\]

It remains to estimate the probability that the label span of a fixed $(k+1)$-set has dimension $k-1$, since this is precisely the excluded event on the upper layer.  Applying~\eqref{eqRankcount} to the matrix whose columns are its labels, with $q=4$, $m=k$, $s=k+1$, and $r=k-1$, and using
\[
\genfrac{[}{]}{0pt}{}{k}{k-1}_4=\frac{4^k-1}{4-1},
\]
gives
\begin{align*}
\Pr(\dim W(A)=k-1)
&=\frac{1}{4^{k(k+1)}}\genfrac{[}{]}{0pt}{}{k}{k-1}_4
\prod_{i=0}^{k-2}(4^{k+1}-4^i)\\
&=\frac{1}{4^{k(k+1)}}\cdot \frac{4^k-1}{3}
\prod_{i=0}^{k-2}(4^{k+1}-4^i)\\
&=\frac1{12}(1-4^{-k})\prod_{i=0}^{k-2}(1-4^{i-k-1})\\
&=\frac1{12}(1-4^{-k})\prod_{j=3}^{k+1}(1-4^{-j}).
\end{align*}
Letting $k\to\infty$, this tends to
\[
\frac1{12}\prod_{j=3}^{\infty}(1-4^{-j})
=\frac1{12}\cdot \frac{\Delta_4}{(1-4^{-1})(1-4^{-2})}
=\frac{16}{135}\Delta_4.
\]
Since $\mathcal U_\phi$ is the complement of this event on the upper layer, \eqref{eqUdensity} follows.
\end{proof}

By linearity of expectation and Lemma~\ref{lemDensities},
\begin{align*}
\mathbb E|\mathcal F_\phi|
&=\binom n{k-1}\left(\frac43\Delta_4+o(1)\right)
+\binom nk\left(1-\Delta_4+o(1)\right)\\
&\qquad +\binom n{k+1}\left(1-\frac{16}{135}\Delta_4+o(1)\right).
\end{align*}
Since $k=\lfloor n/2\rfloor$,
\[
\binom n{k-1}=(1+o(1))\binom nk,
\qquad
\binom n{k+1}=(1+o(1))\binom nk.
\]
Therefore
\begin{align*}
\mathbb E|\mathcal F_\phi|
&=\left(\frac43\Delta_4+1-\Delta_4+1-\frac{16}{135}\Delta_4+o(1)\right)\binom nk\\
&=\left(2+\frac{29}{135}\Delta_4+o(1)\right)\binom nk.
\end{align*}
Hence some labelling $\phi$ satisfies
\[
|\mathcal F_\phi|\ge
\left(2+\frac{29}{135}\Delta_4+o(1)\right)\binom nk.
\]
By Proposition~\ref{propQfree}, this family is $Q_2$-free.  Since $k=\lfloor n/2\rfloor$, this proves Theorem~\ref{thmMain}.

\begin{remark}
The same construction works over every finite field $\F_q$.  Writing
\[
\Delta_q=\prod_{i=1}^{\infty}(1-q^{-i}),
\]
the resulting leading coefficient is
\[
c_q=2+\frac{q^3-2q^2-q+1}{(q-1)^2(q^2-1)}\Delta_q.
\]
A direct computation shows that $c_q$ is maximized among prime powers at $q=4$.
\end{remark}

\section{Use of automated tools}

The main construction in this manuscript was developed through a discussion with ChatGPT 5.5.  All arguments have been verified carefully by the author, who takes full responsibility for the content of the manuscript.

\end{document}